\documentclass[12pt]{article}
\usepackage{amsmath, amsfonts, amsthm, amssymb, verbatim}
\usepackage{graphicx}
\usepackage[mathcal]{euscript}
\usepackage{amstext}
\usepackage{amssymb,mathrsfs}
\usepackage{bbm}

\usepackage[latin1]{inputenc}
\newcommand{\R}{\mathbb R}
 
 \newcommand{\N}{\mathbb N}

  \newcommand{\E}{\mathbb E}

\newcommand{\ve}{\varepsilon}

\newcommand \loc    {\text{loc}}

\newcommand{\dive}{{\rm div}}

\newtheorem{theorem}{Theorem}[section]
\newtheorem{proposition}[theorem]{Proposition}
 \newtheorem{remark}[theorem]{Remark}
\newtheorem{lemma}[theorem]{Lemma}

\newtheorem{definition}[theorem]{Definition}
\newtheorem{hypothesis}[theorem]{Hypothesis}

\begin{document}

\title{ $W^{1,p}$-solutions of the transport equation by stochastic perturbation. }

\author{David A.C. Mollinedo \footnote{Universidade Tecnol\'{o}gica Federal do Parana, Brazil. E-mail:  { \sl davida@utfpr.edu.br}
}}
%

\date{}

\maketitle

\textit{Key words and phrases.
 Stochastic transport equation, Stochastic characteristic method, Regularization.}



%
\begin{abstract}
We consider the stochastic transport equation with a possibly unbounded H\"older continuous vector field. Well-posedness is proved, namely, we show existence, uniqueness
and strong  stability of $W^{1,p}$-weak solutions.
 \end{abstract}
%
\maketitle

%

\section {Introduction} \label{Intro}



During decades the transport equation has attracted a lot of scientific interest. The main reason is due that several physical phenomena in fluid dynamics and kinetic equations can be modeled by the transport equation,
\begin{equation}\label{trasports}
    \partial_t u(t, x) +  b(t,x) \cdot  \nabla u(t,x)  = 0 \,,
\end{equation}
which is one of the most fundamental and at the same time the most elementary partial differential equation with applications in a wide range of problems from physics, engineering, biology, population dynamics or social science. See for instance  Lions' books \cite{lion1} and \cite{lion2} for application in  to fluid dynamics, and to see also Dafermos' book \cite{Dafermos} for more general applications of the transport equation in the domain of conservation laws.

In view that we are interested to study the Cauchy problem \eqref{trasports}, but in the stochastic case, we are going to briefly recall some of the main recent results respect this equation. Di Perna and Lions \cite{DL} have introduced the notion of renormalized solution to the equation \eqref{trasports}, that is, for every  $\beta$ suitable, $\beta(u)$ it is a solution such that
\begin{equation}\label{renord}
    \partial_t \beta (u(t, x) )+  b(t,x) \cdot  \nabla \beta ( u(t,x))  = 0 .
\end{equation}
We observe that (\ref{renord}) holds for smooth solutions, by an immediate application of the chain-rule. However, when the vector field is not
smooth, we cannot expect any regularity of the solutions, so that (\ref{renord}) is a nontrivial request when made for all bounded distributional solutions. This notion of renormalization motivates that, if the renormalization property holds, then solutions of (\ref{trasports}) are unique and stable. In this way, Di Perna and Lions \cite{DL} proved that $W^{1,1}$ spatial regularity of the vector field $b(t,x)$ (with an additional condition of boundedness on the divergence) is sufficient to ensure uniqueness of weak solutions. Later, Ambrosio \cite{ambrosio} proved uniqueness for the case of $BV$ regularity for b instead of $W^{1,1}$. In both results, the uniqueness was based on the commutator ideas. See  Ambrosio and Crippa \cite{ambrosio2} and De Lellis \cite{lellis} for a nice review on that. There are generalizations of these results, but not so far from them; for instance we can mention the works of Alberti, Bianchini and Crippa \cite{ABC} and Hauray \cite{MH} both in 2 dimensions, where the drift does not have any differentiability regularity, but with some additional geometrical conditions.

The last years, a great deal of attention has been focused on the study of stochastic transport equation
\begin{equation}\label{trasport}
 \left \{
\begin{aligned}
    &\partial_t u(t, x) +  \,   \, \big(b(t, x) + \frac{d B_{t}}{dt} \big) \nabla u(t, x) = 0,
    \\[5pt]
    &u|_{t=0}=  u_{0},
\end{aligned}
\right .
\end{equation}
$\big( (t,x) \in U_T, \omega \in \Omega \big)$, where $U_T= [0,T] \times \R^d$, for $T>0$ be any fixed real number, $(d \in \N)$, $b:\R_+ \times
\R^d \to \R^d$ is a given vector field (drift), with $\dive\, b=0$, $B_{t} = (B_{t}^{1},...,B _{t}^{d} )$ is a standard Brownian motion in $\mathbb{R}^{d}$ and the stochastic integration is taken (unless otherwise mentioned) in the Stratonovich sense.

\bigskip

We observe that, there are several situations where the stochastic problem has better behavior than deterministic one. A first result in this direction was given by Flandoli, Gubinelli and Priola \cite{FGP2}, where they obtained well-posedness of $L^\infty$-weak solutions of the Cauchy problem \eqref{trasport} for the H\"older continuous and bounded drift term, with some integrability conditions on the divergence and initial condition $u_0\in L^\infty(\R^d)$. The key tool of this work has been a differentiable stochastic flow (characteristics method) constructed and analysed by means of a special transformation of the drift of Itô-Tanaka type. Similarly, Fedrizzi and Flandoli \cite{Fre1} obtained a well-posedness result of ``weakly differentiable solutions'' under only some integrability conditions on the drift term and considering the initial datum $u_0\in \cap_{r\geq 1} W_{loc}^{1,r}(\R^d)$. Specifically, they only assumed that
\begin{equation}\label{LPSC}
    \begin{aligned}
    &b\in L^{q}\big( [0,T] ; L^{p}(\mathbb{R}^{d}) \big) \, , \\[5pt]
 \mathrm{for}  \qquad &\  p,q \in [2,\infty) \, , \qquad   \qquad  \frac{d}{p} + \frac{2}{q} < 1 \, .
\end{aligned}
\end{equation}
Later, Neves and Olivera \cite{NO} and \cite{NO1} under condition \eqref{LPSC} got, for measurable initial condition, existence and uniqueness of $L^\infty$-weak solutions for the stochastic transport/continuity equation \eqref{trasport}. In this work, the authors have used the main feature of the transport equation, which is the transportation property, to show uniqueness in a different way from the renormalization idea (which exploit commutators) used in, for example, \cite{Cat-Oli} and \cite{FGP2}. In 2013, Catuogno and Olivera \cite{Cat-Oli} proved existence and uniqueness of $L^p$-solutions for the problem \eqref{trasport} with initial condition $u_0\in L^p(\R^d)$. Here, the authors used the generalized Itô-Ventzel-Kunita formula (see Theorem 8.3 of \cite{Ku2}) and the results on existence and uniqueness for the deterministic transport linear equation (see for instance \cite{DL}). In another context Fedrizzi, Neves and Olivera \cite{Fre2}, working with the notion of ``quasiregular weak solutions'', obtained uniqueness of \eqref{trasport} when the field vectors $b\in L_{loc}^{2}$. We also mention that Mollinedo and Olivera \cite{Moli} showed uniqueness of $L^2$-weak solutions for one-dimensional stochastic transport/continuity equation with unbounded measurable drift without assumptions on the divergence. More precisely, they only assumed the vector field $b$ satisfies
$$|b(x)|\leq k(1+|x|)\,;$$
then, to prove uniqueness they have used the fact that ``one primitive of a $L^2$-weak solution'' is regular and verifies the stochastic transport equation \eqref{trasport}.  Finally, Mollinedo and Olivera \cite{Moli1} obtained well-posedness of \eqref{trasport}  with unbounded drift but in the context of weighted spaces.

\medskip

The contribution of the present paper is to prove {\it uniqueness of $W^{1,p}$-weak solutions} (see Definition \ref{defisolu1}) of the Cauchy problem \eqref{trasport} for H\"older continuous, possibly unbounded,  divergence-free drift. In particular, this result implies the {\it persistence of regularity} for initial conditions $u_0\in W^{1,p}(\R^d)$, with $1<p<\infty$. Here, as our drift term $b$ is unbounded locally H\"older continuous, then we have suitable regularity of the stochastic flow associated to this vector field $b$ (see \cite[Theorem 7]{FGP}). Thereby in the proof of our main results, using a regularization procedure, thanks to the free-divergence condition, following the same strategy introduced in \cite{Moli1} and in order to avoid ``commutators'' and the problems there in, we are able to compose the solution $u$ of the stochastic transport equation with the stochastic flow, in fact its inverse; in this way, we can bring the stochastic flow with all its space derivatives on the test function.

\medskip

In the context of the present paper we would like to point and remark the following:
\begin{enumerate}
  \item For the deterministic transport equation the problem $W^{1,p}$-solution is open, under essentially weaker conditions than Lipschitz continuity of the vector field $b$. Besides that, we would like to mention that Colombini, Luo and Rauch \cite{Colom} proved, through a specific example in $\R^2$ and considering divergence free vector fields $b\in \cap_{1\leq p< \infty} W^{1,p}(\R^2)$, that the persistence property for the deterministic case is not true even the uniqueness is established. Specifically, one may start with an initial data $u_0\in C_c^{\infty}(\R^2)$, but the deterministic unique bounded solution is not continuous on any neighborhood of the origin.
	\item Compared to \cite{FGP3}, the approach of our work to prove the uniqueness result is different because in \cite{FGP3} the authors used ideas based in ``commutator lemma'' (see  for example \cite{DL} and \cite{FGP2}) and as it was mentioned previously we avoid considering  ideas on ``commutators''. Moreover, in \cite[Theorem 6 and 7]{FGP3} the authors worked with globally H\"older continuous and bounded vector field $b\in L^{\infty}(0,T; C_b^{\alpha}(\R^d;\R^d))$  with $\alpha\in (0,1)$ and initial condition $u_0\in L^{\infty}(\R^d)$, which is not our case (see Hypothesis \ref{hyp}). Now, another difference with \cite{FGP3} is related to the persistence property. In the work \cite[Theorem 4]{FGP3} was showed the persistence of $C^1$-regularity but in the present paper we prove persistence of $W^{1,p}$-regularity with $1<p<\infty$.
	\item Now, concerning the work \cite{Moli1} to prove well-posedness of the SPDE \eqref{trasport} the authors worked in the context of weighted spaces. Indeed, they worked in the space $L^{2p}(\Omega \times [0,T]\times \R^d)\cap L^p(\Omega \times [0,T], W^{1,p}(\R^d), \mu)$ where the weight $\mu$ is the Gaussian measure in $\R^d$  defined as $\mu=e^{-|x|^2}$ (see \cite[Definition 2.1]{Moli1}). In the present paper, to show the main results, using the strategy introduced in \cite{Moli1} and thanks to the free-divergence condition on the vector field, we avoid to work in the context of weighted spaces.
	\item Finally, we would like to emphasize we are not interested in to show estimates on the flow associated to vector field $b$ because, under our conditions (see \eqref{con1}), we have good estimates that have been  already investigated in \cite[Theorem 7]{FGP}. So, we only use these properties to prove well-posedness and the persistence property of the Cauchy problem  \eqref{trasport}.
\end{enumerate}


\medskip

Throughout of this paper, we fix a stochastic basis with a
$d$-dimensional Brownian motion $\big( \Omega, \mathcal{F}, \{
\mathcal{F}_t: t \in [0,T] \}, \mathbb{P}, (B_{t}) \big)$.

\section{Preliminaries}

To establish well-posedness for the stochastic partial differential equation (SPDE) \eqref{trasport} we need to assume the following hypothesis:

\begin{hypothesis}\label{hyp}
The drift term is taken to be
\begin{equation}\label{con1}
    b\in C^{\theta}(\R^d,\R^d)\,,
\end{equation}
and
\begin{equation}\label{con2}
 \dive \, b =0\, .
\end{equation}
Besides that, for $1<p<\infty$ the initial condition satisfies
\begin{equation*}\label{conIC}
 u_0 \in  W^{1,p}(\R^d)   \, .
\end{equation*}
\end{hypothesis}

\subsection{Notations}

For any  $\theta \in (0,1)$, we denote   $C^{\theta}(\R^d;\R^d)$, $d \geq 1$ the space of the vector fields  $f:\R^d \to \R^d$  for which
\begin{align*}
[f]_{\theta}:=\sup_{x\neq y, |x-y|\leq 1} \frac{|f(x)-f(y)|}{|x-y|^\theta} < \infty\,.
\end{align*}

The space $C^{\theta}(\R^d;\R^d)$ becomes a Banach space with the norm
\begin{align}\label{eq-t2-norma para C theta}
\|f\|_{\theta}=\|(1+|\cdot|)^{-1}f(\cdot)\|_\infty+[f]_{\theta}\,,
\end{align}
where $\|\cdot\|_{\infty}$ denotes the supremum norm over $\R^d$.

\medskip

Now, we are going to recall some main results about stochastic flows. Thus, for $0\leq s\leq t$ and $x\in\mathbb{R}^{d}$, consider the following
stochastic differential equation (SDE) in $\mathbb{R}^{d}$ associated to the vector field $b$
\begin{equation}
\label{itoass}X_{s,t}(x)= x + \int_{s}^{t} b( X_{s,r}(x)) \ dr + B_{t}-B_{s}\,,
\end{equation}
where $X_{s,t}(x)= X(s,t,x)$ and $X_{t}(x)= X(0,t,x)$. Under condition (\ref{con1}), the process $X_{s,t}(x)$ is a stochastic flow of $C^{1}$-diffeomorphism (see Flandoli, Gubinelli and Priola \cite{FGP}). Moreover, the inverse flow $Y_{s,t}(x):=X_{s,t}^{-1}(x)$ satisfies the following backward stochastic differential equation
\begin{equation}
\label{itoassBac}Y_{s,t}(x)= x - \int_{s}^{t} b( Y_{r,t}(x)) \ dr - (B_{t}-B_{s}),
\end{equation}
for $0\leq s\leq t$, see Flandoli, Gubinelli and Priola \cite{FGP2}. Usually $Y$ is called the time reversed process of $X$.

\medskip

From Flandoli, Gubinelli and Priola \cite[Theorem 7]{FGP}, we also remember the following result that we are going to use in our main results: Let  $b_n\in C^{\theta}(\R^d,\R^d)$ and let $\phi^{n}$ be the corresponding stochastic flows. Assume that there exists $b\in C^{\theta}(\R^d,\R^d)$ such that $b_n-b \in C_b^{\theta}(\R^d,\R^d)$, $n\geq 1$, and $\|b_n -b\|_{C_b^{\theta}(\R^d,\R^d)}\rightarrow 0 $ as $n\rightarrow \infty$. If $\phi$ is the flow associated to $b$, then for all $p\geq 1, T>0,$ we have

\begin{equation}\label{est1}
 \lim_{n\rightarrow\infty}  \sup_{x \in \R^d}   \sup_{s\in[0,T]} \mathbb{E}[  \sup_{t \in [s,T]}  |D\phi_{s,t}^{n}(x)- D\phi_{s,t}(x)|^p] =0 ,
\end{equation}

\begin{equation}\label{est2}
 \lim_{n\rightarrow\infty}\sup_{x \in \R^d} \sup_{s \in [0,T]}\mathbb{E}\Big[  \sup_{t \in [s,T]}  \Big|\frac{\phi_{s,t}^{n}(x)- \phi_{s,t}(x)}{ 1+ |x|}\Big|^p \Big] =0 ,
\end{equation}
and
\begin{equation}\label{est3}
 \sup_{n}  \sup_{x \in \R^d}   \sup_{s\in[0,T]} \mathbb{E}[  \sup_{t \in [s,T]}  |D\phi_{s,t}^{n}(x)|^p] <  \infty.
\end{equation}
If we denote by $\psi$ the inverse flow of $\phi$, then same results are valid for the backward flows $\psi_{s,t}^{n}$ and  $\psi_{s,t}$ since
are solutions of the same SDE driven by the drifts $-b_{n}$ and $-b$, respectively.

\subsection{Definitions of weak solutions}

We present now a suitable definition of $W^{1,p}$-weak solution to equation  \eqref{trasport} to treat the problem of well-posedness under our hypothesis \ref{hyp}. We denote $ C_c^{\infty}(\R^d)$ the space of the  test functions with compact support.

\begin{definition}\label{defisolu1}
 A stochastic process $u\in L^{p}( \Omega\times[0, T], W^{1,p}(\mathbb{R}^{d}))$ is called a $W^{1,p}$-weak solution of the Cauchy problem \eqref{trasport} when:
 for any $\varphi \in C_c^{\infty}(\R^d)$, the real valued process $\int  u(t,x)\varphi(x)  dx$ has a continuous modification which is an
$\mathcal{F}_{t}$-semimartingale, and for all $t \in [0,T]$, we have $\mathbb{P}$-almost surely
\begin{equation} \label{DISTINTSTR}
\begin{aligned}
    \int_{\R^d} u(t,x) \varphi(x) dx = &\int_{\R^d} u_{0}(x) \varphi(x) \ dx
   -\int_{0}^{t} \!\! \int_{\R^d}     \, b^{i}(x) \partial_i u(s,x) \varphi(x) \ dx ds
\\[5pt]
    &+ \int_{0}^{t} \!\! \int_{\R^d} u(s,x) \ \partial_{i} \varphi(x) \ dx \, {\circ}{dB^i_s} \, .
\end{aligned}
\end{equation}
\end{definition}


\begin{remark}\label{lemmaito}
From the idea of Flandoli, Gubinelli and Priola \cite[Lemma 13]{FGP2}, we can write the problem (\ref{trasport}) in It\^o's form as follows. A  stochastic process  $u\in L^{p}( \Omega\times[0, T] ,  W^{1,p}(\mathbb{R}^{d}))$ is  a  $W^{1,p}$-weak solution  of the stochastic transport equation (\ref{trasport}) iff for every test function $\varphi \in C_{c}^{\infty}(\mathbb{R}^{d})$, the process $\int u(t, x)\varphi(x) dx$ has a continuous modification which is a $\mathcal{F}_{t}$-semimartingale and satisfies the following It\^o's formulation
\begin{align*}
\int u(t,x) \varphi(x) dx & = \int u_{0}(x) \varphi(x) \ dx  -\int_{0}^{t} \int   b^{i}(x) \partial_i u(s,x)   \varphi(x)  \ dx\, ds \\[5pt]
& \quad+  \int_{0}^{t} \int  \partial_{i}\varphi(x) u(s,x) \ dx \ dB_{s}^{i}
\end{align*}
\begin{align*}
&\quad\quad\quad+\frac{1}{2}   \int_{0}^{t} \int  \Delta \,\varphi(x) u(s,x) \ dx\, ds  .
\end{align*}
\end{remark}

\section{Main Results}

\subsection{ Existence of weak solutions}

In this section, we will prove existence of $W^{1,p}$-weak solution under hypothesis \ref{hyp}. The key points of the proof is the ``regularization procedure'' and the inequalities \eqref{est2} and \eqref{est3}.

\begin{lemma}\label{renor}  Assume hypothesis (\ref{hyp}). Then there exists a $W^{1,p}$-weak solution $u$ of the Cauchy problem \eqref{trasport} and $u(t,x)=u_{0}(\phi_t^{-1}(x))$.
\end{lemma}

\begin{proof}
We divide the proof in two steps.

\medskip

{\it Step 1 : Assume $u_0 \in C_c^{\infty}(\R^d)$}. From a minor modification of the arguments in Mollinedo and Olivera \cite[Proposition 2.3]{Moli1} it follows that $u(t,x)=u_{0}(\phi_t^{-1}(x))$ satisfies
\begin{align}\label{eq-t1-regul renormalizada com beta-itobis}
\int_{\R^d} u(t,x) \varphi(x) dx    = &\int_{\R^d} u_0(x)\varphi(x) \ dx \nonumber\\[5pt]
                                                   & -\int_{0}^{t} \int_{\R^d} \partial_i u(s,x) \  b^{i}(x) \ \varphi(x) \ dx \,ds \nonumber\\[5pt]
                                                                                  	& + \int_{0}^{t} \int_{\R^d} u(s,x) \ \partial_{i} \varphi(x) \ dx \,{dB^i_s} \nonumber\\[5pt]
                                                   & + \frac{1}{2} \int_{0}^{t} \int_{\R^d} u(s,x) \Delta\varphi(x)\,dx\, ds\,.
\end{align}
As $\dive \ b=0$, that is, the Jacobian of the stochastic flow is identically one, we observe that making the change of variables $y=\psi_t(x)=\phi_t^{-1}(x)$ we obtain
\begin{align*}
\int_0^T \int_{\Omega}\int |  u_0 (\psi_t(x)) |^{p} \ dx \mathbb{P}(d\omega) dt & =T \int |  u_0 (y) |^{p} dy  \\[5pt]
&\leq T \|u_0\|^p_{W^{1,p}(\R^d)}
\end{align*}
and
\begin{align}\label{eq1}
\int_0^T \int_{\Omega}\int_{\R^d} & | D u_0 (\psi_t(x)) |^{p} \  dx \mathbb{P}(d\omega) dt \nonumber\\[5pt]
&=\int_0^T \int_{\Omega}\int_{\R^d} | D u_0 (\psi_t(x)) |^{p} |D\psi_t(x)|^p \ dx \mathbb{P}(d\omega) dt \nonumber\\[5pt]
&\leq \,\,\int_0^T \int_{\R^d} | D u_0 (y) |^{p} \,\mathbb{E}[|D\psi_t(\phi_t(y),\omega)|^p ]\ dy (d\omega) dt\,.
\end{align}
Now, we see that
\begin{align}\label{eq-flow-flow}
D(\psi_t (\phi_{t}(x),\omega))= D^{-1}(\phi_{t}(x))
\end{align}
and
\begin{align}\label{eq-cofator}
D^{-1}(\phi_{t}(x))= Cof (D \phi_{t}(x))^{T}\,,
\end{align}
where $Cof$ denotes the cofactor matrix of $D \phi_{t}$. By inequality \eqref{est3} we get that $Cof (D \phi_{t}(x))^{T}\in L^{\infty}(\R^d; L^{p}(\Omega; L^{\infty}([0,T])))$. Thus, considering \eqref{eq1} we have
\begin{align*}
\int_0^T \int_{\Omega}\int_{\R^d} & | D u_0 (\psi_t(x)) |^{p} \ dx \mathbb{P}(d\omega) dt \\[5pt]
&= \,\,\int_0^T \int_{\R^d} | D u_0 (y) |^{p} \,\mathbb{E}[|D\psi_t(\phi_t(y),\omega)|^p ]\ dy \ dt\,\\[5pt]
& \leq C \int_{\R^d} | D u_0 (y) |^{p} dy\,\\[5pt]
& \leq C \|u_0\|^p_{W^{1,p}(\R^d)}
\end{align*}
Therefore, we conclude  that $u(t,x)=u_0 (\phi^{-1}_t(x))$ is a $W^{1,p}$-weak solution of the equation (\ref{trasport}).

\medskip

{\it Step 2: Assume $u_0 \in  W^{1,p}(\R^d)$}. Let $\{\rho_\varepsilon\}_\varepsilon$ be a family of standard symmetric mollifiers and $\eta$ a non-negative smooth cut-off function supported on the ball of radius 2 and such that $\eta=1$. So, for every $\varepsilon>0$, we introduce te rescaled functions $\eta_{\varepsilon} (\cdot) =  \eta(\varepsilon \cdot)$. In this way, we define the family of regularized initial conditions given by
$$u_0^\varepsilon (x) = \eta_\varepsilon(x) \big( [ u_0(\cdot) \ast \rho_\varepsilon (\cdot) ] (x)  \big)\,.$$
From step 1 it follows that $u^\varepsilon(t,x) =u_0^\varepsilon(\phi^{-1}_t(x))$ verifies (Itô's form)
\begin{align}\label{eqirr}
\int_{\R^d} u^\epsilon(t,x) \varphi(x) dx    = &\int_{\R^d} u_0^\varepsilon(x)\varphi(x) \ dx \nonumber\\[5pt]
                                                   & -\int_{0}^{t} \int_{\R^d} \partial_i u^\epsilon(s,x) \ b(x) \ \varphi(x) \ dx \,ds \nonumber\\[5pt]
                                               & + \int_{0}^{t} \int_{\R^d} u^\epsilon(s,x) \ \partial_{i} \varphi(x) \ dx \,{dB^i_s} \nonumber\\[5pt]
                                                   & + \frac{1}{2} \int_{0}^{t} \int_{\R^d} u^\epsilon(s,x) \Delta\varphi(x)\,dx\, ds\,.
\end{align}
Now, we can see that $u^\epsilon(s,x)=u_0^{\varepsilon}(\phi^{-1}_t(x))$ converges strongly to $u(s,x)=u_0(\phi^{-1}_t(x))$ in  $L^p(\Omega\times[0,T],  W^{1,p}(\R^{d}))$. In fact, doing the change of variables $y=\psi_t(x)=\phi^{-1}_t(x)$  we have
\begin{align*}
\int_{\Omega}\int_{0}^{T} \int_{\R^d} |u_0^{\varepsilon}(\psi_t(x)) -  u_0(\psi_t(x))|^{p}\ dx \ ds \mathbb{P}(d\omega) & = \int_{0}^{T} \int_{\R^d}   |u_0^\epsilon(y) -  u_0(y)|^{p} \ dy \ ds\,,
\end{align*}
 and using \eqref{eq-cofator}, \eqref{eq-flow-flow}  and \eqref{est3} it results
\begin{align*}
\int_{\Omega}\int_{0}^{T} \int_{\R^d}  & | D u_0^{\varepsilon}(\psi_t(x))- Du_0(\psi_t(x))|^p \ dx \mathbb{P}(d\omega) dt \\[5pt]
&\leq \int_{\Omega}\int_{0}^{T} \int_{\R^d}  | D u_0^{\varepsilon}(\psi_t(x))- Du_0(\psi_t(x))|^p |D\psi_t(x)|^p  \ dx \mathbb{P}(d\omega) dt \\[5pt]
& \leq \int_0^T \int_{\R^d} | D u_0^{\varepsilon}(y)- Du_0(y)|^p \ \mathbb{E}[|D\psi_t(\phi_t(y),\omega)|^p ] \ dy  dt \\[5pt]
& \leq C \int_{\R^d} | D u_0^{\varepsilon}(y)- Du_0(y)|^p dy
\end{align*}
Therefore, by the calculus made above, we can pass to the limit in (\ref{eqirr}) as $\varepsilon \to 0$.  Thus, we conclude the that $u(t,x)=u_0 (\phi^{-1}_t(x))$ is a $W^{1,p}$-weak solution of SPDE (\ref{trasport}).
%
%
%

%
\end{proof}

\subsection{Uniqueness of weak solutions}\label{UNIQUE}

In the present section, we shall show the uniqueness result for $W^{1,p}-$weak solutions for the SPDE \eqref{trasport}. As mentioned in the introduction, using the divergence-free condition, in the demonstration below we avoid to consider ideas on ``commutators''.
\medskip

\begin{theorem}\label{uni}
Under the conditions of hypothesis (\ref{hyp}) uniqueness holds for $W^{1,p}-$weak solutions of the Cauchy problem \eqref{trasport} in the following sense: If $u_1,u_2 \in L^p([0,T] \times \Omega, W^{1,p}(\mathbb{R}^{d}))$ are two $W^{1,p}-$weak solutions with the same initial data $u_{0}\in W^{1,p}(\mathbb{R}^{d})$, then  $u_1= u_2$ almost everywhere in $[0,T] \times\R^d \times \Omega$.
\end{theorem}

\begin{proof}
We divide the proof in three steps. Before starting the proof we see that, by linearity, is sufficient to prove that a $W^{1,p}$-weak solution with initial condition $u_0=0$ vanishes identically. Let us denote  by $u$ such a solution.

{\it Step 1: Regularization.} For $\varepsilon>0$ and $\delta>0$, let us consider $\phi_{\varepsilon}, \phi_{\delta}$ the standard symmetric mollifiers. So, by considering $u_{\varepsilon}(t,\cdot)=u(t,\cdot)\ast \phi_{\varepsilon}$ we get the integral equation
\begin{equation} \label{DISTINTSTRAPP}
\begin{aligned}
     u_\varepsilon(t,x)  = & -\int_{0}^{t} \!\! \int_{\R^d} \partial_i u(s,z) \ b^i(z)  \phi_\varepsilon(x-z) \ dz ds
\\[5pt]
    &  + \int_{0}^{t} \!\! \int_{\R^d} u(s,z) \ \partial_{i} \phi_\varepsilon(x-z) \ dz \, {\circ}{dB^i_s}\,,
\end{aligned}
\end{equation}
which, for each $\ve> 0$, is strong  in the analytic sense.

\medskip

Now, let us denote by $b^{\delta}$ and $X_t^{\delta }$  the standard mollification of $b$ and the associated flow to the SDE (\ref{itoass}) (with $b^{\delta}$ instead $b$), respectively. Analogously, let $Y^\delta_{t}$ be the associated flow to the SDE \eqref{itoassBac}. In this way, by a change of variables $x=X_t^{\delta}(y)$ and for each $\varphi \in C^{\infty}_c(\R^d)$ we see
\begin{equation}
\label{PUSHFORWARD}
    \int_{\R^d}  u_{\varepsilon}(t, X_t^{\delta}(y)) \; \varphi(y) \ dy
    =\int_{\R^d}  u_{\varepsilon}(t,x) \;   \varphi(Y_t^{\delta}(x)) \ dx \,,
\end{equation}
for each $t \in [0,T]$ (recall that $\dive \ b^{\delta} =0$, that is, the Jacobian of the stochastic flow is identically one). We also see that by applying  It\^{o}'s formula (see \cite{Ku3}) to the process $v^{\delta}(t,x)= \varphi(Y_t^{\delta})$ it satisfies the stochastic transport equation in the classical sense, that is,
\begin{equation}\label{trasportC}
 \left \{
\begin{aligned}
    & dv^{\delta}(t, x) + b^{\delta}(x) \nabla v^{\delta}(t, x)   \ dt + \nabla v^{\delta}(t, x)  \circ  d B_{t}= 0,
    \\[5pt]
    & v^{\delta}|_{t=0}=  \varphi(x).
\end{aligned}
\right .
\end{equation}
As $u_\ve$ is strong in the analytic sense, then, by applying again It\^o's formula to the product of semimartingales $$u_{\varepsilon}(t,x)  \varphi(Y_t^{\delta}),$$
we have
\begin{align}\label{eq-regular}
u_{\varepsilon}(t,x) \;  \varphi(Y_t^{\delta}) =& -\int_{0}^{t}\!\!  u_\varepsilon(s,x)  \; b^{\delta}(x) \nabla \varphi(Y_t^{\delta})   \  ds - \int_{0}^{t}  u_\varepsilon(s,x) \;   \partial_{x_i} [\varphi(Y_t^{\delta})] \  \circ dB_{s}^{i} \nonumber\\[5pt]
& -\int_{0}^{t} \!\! \int_{\R^d} \varphi(Y_t^{\delta})    \int_{\R^d}   \partial_{y_i} u(s,y) \; b^{i}(y) \  \phi_{\varepsilon}(x-y)  \ dy \ ds \nonumber\\[5pt]
&+  \int_{0}^{t} \!\! \int_{\R^d}  \varphi(Y_t^{\delta})  u(s,y) \; \partial_{y_i} \phi_{\varepsilon}(x-y)  \ dy \  \circ dB_{s}^{i}\,.
\end{align}

{\it Step 2: Localization and Passing to the limit as $\ve \to 0$ and $\delta\to 0$.} The idea in this step is to pass to the limit in the above equation \eqref{eq-regular}, but under condition \eqref{con1} it is not possible because the Dominated Convergence Theorem does not work. Thus, in order to avoid that problem initially we need to do a ``Localization''.

\medskip

In fact, multiplying the equation \eqref{eq-regular} by a differentiable function $\eta_R$ defined by $\eta_R(x)=\eta(x/R)$ ($R>0$), where $\eta\in C_c^{\infty}(\R^d)$ is a cut-off function such that $0\leq \eta\leq 1$, $\eta\equiv 1$ in the ball of radius one and with support in the ball of radius 2; and integrating on $\R^d$ we have
\begin{align*}
\int_{\R^d}  u_{\varepsilon}(t,x) & \;  \varphi(Y_t^{\delta}) \ \eta_R(x)\ dx =-\int_{0}^{t}\!\!  \int_{\R^d}  u_\varepsilon(s,x)  \; b^{\delta}(x) \cdot \nabla \varphi(Y_t^{\delta}) \eta_R(x)  \ dx  \  ds \\[5pt]
& - \int_{0}^{t}  \!\! \int_{\R^d}  u_\varepsilon(s,x) \;   \partial_{x_i} [\varphi(Y_t^{\delta})]  \eta_R(x) dx \  \circ dB_{s}^{i} \\[5pt]
& -\int_{0}^{t} \!\! \int_{\R^d} \varphi(Y_t^{\delta}) \ \eta_R(x)   \int_{\R^d}   \partial_{y_i} u(s,y) \; b^{i}(y) \  \phi_{\varepsilon}(x-y) \ dy \ dx \ ds \\[5pt]
& + \int_{0}^{t} \!\! \int_{\R^d}  \varphi(Y_t^{\delta}) \ \eta_R(x) \int_{\R^d} u(s,y) \; \partial_{y_i} \phi_{\varepsilon}(x-y) \ dy \ dx \  \circ dB_{s}^{i}.
\end{align*}
Integrating by parts the last term on right hand side of the above equation and using that $\partial_{y_i} \phi_{\varepsilon}(x-y)=-\partial_{x_i} \phi_{\varepsilon}(x-y)$ it results
\begin{align*}
\int_{\R^d}  u_{\varepsilon}(t,x) & \;  \varphi(Y_t^{\delta}) \ \eta_R(x)\ dx \\[5pt]
& =-\int_{0}^{t}\!\!  \int_{\R^d}  u_\varepsilon(s,x)  \; b^{\delta}(x) \cdot \nabla \varphi(Y_t^{\delta}) \eta_R(x)  \ dx  \  ds \\[5pt]
& \quad+ \int_{0}^{t}  \!\! \int_{\R^d}  u_\varepsilon(s,x) \; \varphi(Y_t^{\delta}) \ \partial_{x_i} \eta_R(x) dx \  \circ dB_{s}^{i} \\[5pt]
& \quad-\int_{0}^{t} \!\! \int_{\R^d} \varphi(Y_t^{\delta})  \eta_R(x)   \int_{\R^d}   \partial_{y_i} u(s,y) \; b^{i}(y) \  \phi_{\varepsilon}(x-y) \ dy \ dx \ ds\,.
\end{align*}
Passing to Itô's formulation, we get
\begin{align}\label{eq-N1-3}
\int_{\R^d}  u_{\varepsilon}(t,x) & \;  \varphi(Y_t^{\delta}) \ \eta_R(x)\ dx =-\int_{0}^{t}\!\!  \int_{\R^d}  u_\varepsilon(s,x)  \; b^{\delta}(x) \cdot \nabla \varphi(Y_t^{\delta}) \ \eta_R(x)  \ dx  \  ds \nonumber\\[5pt]
& + \int_{0}^{t}  \!\! \int_{\R^d}  u_\varepsilon(s,x) \; \varphi(Y_t^{\delta}) \  \partial_{x_i} \eta_R(x) \ dx \   dB_{s}^{i} \nonumber\\[5pt]
&+\frac{1}{2}  \int_{0}^{t}  \!\! \int_{\R^d} u_{\varepsilon}(s,x) \; \varphi(Y_t^{\delta}) \ \Delta\eta_R(x) \ dx \ ds \nonumber\\[5pt]
& -\int_{0}^{t} \!\! \int_{\R^d} \varphi(Y_t^{\delta}) \ \eta_R(x)   \  \partial_{x_i} (u(s,\cdot) \; b(\cdot))_{\varepsilon}(x)  \ dx  \ ds.
\end{align}
Now, for $\delta> 0$ fixed, by the Dominated Convergence Theorem we take the limit as $\varepsilon$ goes to $0^+$ in the above equation to obtain
\begin{align}\label{eq-N3-1}
\int_{\R^d}  u(t,x) & \;  \varphi(Y_t^{\delta}) \ \eta_R(x)\ dx =-\int_{0}^{t}\!\!  \int_{\R^d}  u(s,x)  \; b^{\delta}(x) \cdot \nabla \varphi(Y_t^{\delta}) \eta_R(x)  \ dx  \  ds \nonumber\\[5pt]
& + \int_{0}^{t}  \!\! \int_{\R^d}  u(s,x) \; \varphi(Y_t^{\delta}) \ \partial_{x_i} \eta_R(x) dx \  dB_{s}^{i} \nonumber\\[5pt]
&+\frac{1}{2}  \int_{0}^{t}  \!\! \int_{\R^d} u(s,x) \;  \varphi(Y_t^{\delta}) \ \Delta\eta_R(x) \ dx \ ds \nonumber\\[5pt]
& -\int_{0}^{t} \!\! \int_{\R^d} \varphi(Y_t^{\delta}) \ \eta_R(x) \ \partial_{x_i} (u(s,x) \; b(x))  \ dx \ ds\,.
\end{align}
Again, by using the Dominated Convergence Theorem and the property (\ref{est2}), we can pass to the limit in \eqref{eq-N3-1} as $\delta$ goes to $0^+$, on the space $L^2(\Omega\times[0,T])$, to conclude that
\begin{align}\label{eq-N4-1}
\int_{\R^d}  u(t,x) & \;  \varphi (Y_t(x)) \ \eta_R(x)\ dx =-\int_{0}^{t}\!\!  \int_{\R^d}  u(s,x)  \; b(x) \cdot \nabla \varphi (Y_t(x)) \ \eta_R(x)  \ dx  \  ds \nonumber\\[5pt]
& + \int_{0}^{t}  \!\! \int_{\R^d}  u(s,x) \; \varphi (Y_t(x)) \  \partial_{x_i} \eta_R(x) dx \   dB_{s}^{i} \nonumber\\[5pt]
&+\frac{1}{2}  \int_{0}^{t}  \!\! \int_{\R^d} u(s,x) \;  \varphi (Y_t(x)) \ \Delta\eta_R(x) \ dx \ ds \nonumber\\[5pt]
& -\int_{0}^{t} \!\! \int_{\R^d} \varphi (Y_t(x)) \  \eta_R(x)  \partial_{x_i} (u(s,x) \; b(x))  \ dx \ ds\,.
\end{align}

\medskip

{\it Step 3: Passing to the limit as $R\to +\infty$ and Conclusion.} We are going to prove that, given $\varphi \in C_c^{\infty}(\R^d)$, the $\mathbb{P}$-a.s. limit
$$\int_{\R^d} u(t, X_t(x)) \varphi(x) dx =\lim_{R\to +\infty} \int_{\R^d}  u(t,x)  \; \varphi (Y_t(x)) \ \eta_R(x)\ dx $$
is zero. Hence, from the fact that $X_t$ is a bijection we have that $u$ is identically zero.

\medskip

We start integrating by parts the last term on the right hand side of the equation \eqref{eq-N4-1}:
\begin{align}\label{eq-N5-1}
\int_{\R^d}  u(t,x) & \; \varphi (Y_t(x)) \ \eta_R(x)\ dx = \int_{0}^{t}  \!\! \int_{\R^d}  u(s,x) \; \varphi (Y_t(x))  \partial_{x_i} \eta_R(x) dx \   dB_{s}^{i} \nonumber\\[5pt]
&+\frac{1}{2}  \int_{0}^{t}  \!\! \int_{\R^d} u(s,x) \;  \varphi (Y_t(x)) \ \Delta\eta_R(x) \ dx \ ds \nonumber\\[5pt]
& +\int_{0}^{t} \!\! \int_{\R^d} u(s,x) \ \varphi (Y_t(x))  \; b^i(x) \partial_{x_i} \ \eta_R(x) \ dx \ ds\,.
\end{align}
Now, by considering the estimates (to $R\geq 1$)
\begin{align*}
|\nabla \eta_R(x)|\leq\frac{\|\nabla\eta\|_{L^\infty(\R^d)}}{R} \textbf{I}_{B(R,2R)}(x)
\end{align*}
and
\begin{align*}
|\Delta \eta_R(x)|\leq\frac{\|\Delta\eta\|_{L^\infty(\R^d)}}{R^2} \textbf{I}_{B(R,2R)}(x)
\end{align*}
where $\textbf{I}_A$ represents the indicator function of $A$ and $B(R,2R)=\{x\in \R^d: R<|x|<2R\}$, we see that, doing $R$ goes to $+\infty$, all terms on the right hand side of the equation \eqref{eq-N5-1} converge to zero. In fact, as $\varphi\in C_c^{\infty}(\R^d)$, it is enough to see the following estimates:
\begin{align*}
\int_0^T\int_{\Omega} & \bigg|\int_{0}^{t}  \!\! \int_{\R^d}  u(s,x) \; \varphi (Y_t(x)) \  \partial_{x_i} \eta_R(x) \ dx \   dB_{s}^{i}\bigg|^2 \\[5pt]
& \leq \int_0^T\int_{\Omega} \int_{\R^d}|u(s,x)|^2 \ |\varphi (Y_t(x))|^2 \ |\partial_{x_i} \eta_R(x)|^2 \ dx \ \mathbb{P}(d\omega) ds \\[5pt]
& \leq \frac{\|\nabla\eta\|^2_{L^\infty(\R^d)}}{R^2} \int_0^T\int_{\Omega} \int_{\R^d} |u(s,X_t(y))|^2 \ |\varphi(y)|^2 \textbf{I}_{B(R,2R)}(X_t(y))\ dy \mathbb{P}(d\omega) ds\,.
\end{align*}
Also, if \ $\displaystyle\frac{1}{p}+\frac{1}{q}=1$
\begin{align*}
\bigg|\int_{0}^{t}   &  \int_{\R^d} u(s,x) \;  \varphi (Y_t(x))  \ \Delta\eta_R(x) \ dx \ ds\bigg| \\[5pt]
& \leq \bigg(\int_0^t \int_{\R^d} |u(s,x)|^p |\Delta\eta_R(x)|^p dx \ ds\bigg)^{1/p} \ \bigg(\int_0^t \int_{\R^d} |\varphi (Y_t(x)) |^q  dx \ ds\bigg)^{1/q}  \\[5pt]
& \leq \frac{C}{R^2}\bigg(\int_0^t \int_{\R^d} |u(s,x)|^p \ \textbf{I}_{B(R,2R)}(x)  \ dx \ ds\bigg)^{1/p}\,.
\end{align*}
Likewise, as above
\begin{align*}
\bigg|\int_{0}^{t} \!\! & \int_{\R^d} u(s,x) \ \varphi (Y_t(x))   \; b^i(x) \ \partial_{x_i}\eta_R(x) \ dx \ ds\bigg| \\[5pt]
& \leq \int_{0}^{t} \!\!  \int_{\R^d} |u(s,x)| \ |\varphi (Y_t(x)) |  \; |b^i(x)| \ |\partial_{x_i} \eta_R(x)| \ dx \ ds\\[5pt]
& \leq  \frac{\|\nabla\eta\|_{L^\infty(\R^d)}}{R}(1+2R) \int_{0}^{t} \!\!  \int_{\R^d} |u(s,x)| \ |\varphi (Y_t(x)))| \ \Big(\frac{|b^i(x)|}{1+|x|} \Big)\ \textbf{I}_{B(R,2R)}(x) \ dx \ ds \\[5pt]
& \leq C \|b\|_{C^{\theta}} \bigg(\int_0^t \int_{\R^d} |u(s,x)|^p \ \textbf{I}_{B(R,2R)}(x)  \ dx \ ds\bigg)^{1/p}\,.
\end{align*}
Therefore, passing to the limit in the equation \eqref{eq-N5-1} as $R \to +\infty$ we find out that
\begin{equation}
\label{UNIQ30}
\int_{\R^d} u(t, X_t(x)) \varphi(x) dx =0
\end{equation}
for each $\varphi \in C_c^\infty(\R^d)$, and $t \in [0,T]$.

\bigskip
To conclude, let $K$ be any compact set in $\R^d$. Then, from the above equation we obtain
$$
  \begin{aligned}
    \!\! \int_{K}   \mathbb{E}| u(t,x) |^{p}  \ dx &=   \!\! \int_{K} \mathbb{E}| u(t, X_t(Y_t(x)) ) |^{p}  \ dx
\\[5pt]
     &=   \;  \mathbb{E}   \!\!  \int_{Y_t(K)}    |u(t, X_t(y)) |^{p}  \ dy  = 0\,,
\\[5pt]
\end{aligned}
$$
where it has been used \eqref{UNIQ30} and  that $X_t$ is a stochastic flows of diffeomorphism. Consequently, the conclusion of our  theorem is proved.
\end{proof}

\bigskip

\subsection{ Stability}

In this last section, we prove the strong stability property for the solutions of the stochastic transport equation \eqref{trasport}. Such property  will be shown first with respect to the initial condition $u_0$ and after respect to the vector field $b$.

\begin{proposition}\label{strong}
Assume the conditions \eqref{con1} and \eqref{con2}. Then, for any sequence $\{u_0^n\} \subset  W^{1,p}(\R^d)$ strong converging to $u_0 \in W^{1,p}(\R^d)$ we have:
\begin{align*}
u^n(t,x)\quad\text{converges to}\quad u(t,x)\quad\text{in}\quad L^p([0,T] \times \Omega, W^{1,p}(\mathbb{R}^{d}))
\end{align*}
as $n \to \infty$, where $u^n(t,x)$ and $u(t,x)$ are the unique $W^{1,p}-$weak solutions of the Cauchy problem \eqref{trasport} with the initial data $u_0^{n}$ and $u_0$, respectively.
\end{proposition}


\begin{proof} From Lemma \ref{renor} and Theorem \ref{uni}  we have  that
\begin{align*}
u^n(t,x)=u_0^{n}(X_{t}^{-1})\quad\text{and}\quad u(t,x)=u_0(X_{t}^{-1})\,.
\end{align*}
Then, making $y=X_{t}^{-1}(x)$ we obtain (recall that $\dive \ b =0$, that is, the Jacobian of the stochastic flow $X_t$ is identically one)
\begin{align}\label{estab1}
\int_{0}^{T} \int_{\R^{d}} \E|u^n(t,x)-u(t,x)|^{p} dx & = \int_{0}^{T} \int_{\R^{d}} \E| u_0^{n}(X_{t}^{-1}(x))- u_0(X_{t}^{-1}(x))|^{p} dx \ ds \nonumber\\[5pt]
& = T \int_{\R^{d}}  | u_0^{n}(y)- u_0(y)|^{p}   dy \,.
\end{align}
And considering \eqref{eq-flow-flow}, \eqref{eq-cofator} and \eqref{est3} we have
\begin{align}\label{estab3}
\int_{0}^{T} & \int_{\R^{d}} \mathbb{E}|D u^n(t,x)-D u(t,x)|^{p} \ dx \ ds \nonumber\\[5pt]
& =\int_{0}^{T} \int_{\R^{d}} \mathbb{E}[| Du_0^{n}(X_{t}^{-1}(x))- Du_0(X_{t}^{-1}(x))|^{p} |DX_{t}^{-1}(x)|^{p}] \  dx \ ds \nonumber\\[5pt]
& = \int_{0}^{T} \int_{\R^{d}} | Du_0^{n}(y)- Du_0(y)|^{p} \ \mathbb{E}[|DX_{t}^{-1}(X_t(y),\omega)|^{p}] \  dy \ ds \nonumber\\[5pt]
& \leq C \ \int_{\R^{d}} | Du_0^{n}(y)- Du_0(y)|^{p} dy\,.
\end{align}
So, by the calculus made above  and as $u_0^n$ converges to $u_0$ in $W^{1,p}(\R^d)$ the thesis our proposition follows.
\end{proof}

\begin{proposition}\label{strong}
Let $u_0\in W^{1,p}(\R^d) \cap W^{1,q}(\R^d)$ for some $q>p$. Then for any sequence $\{b_{n}\}\subset C^{\theta}(\R^d,\R^d)$ and $b \in C^{\theta}(\R^d,\R^d)$ satisfying the condition \eqref{con2} such that $\|b_n -b\|_{C_b^{\theta}(\R^d,\R^d)}\rightarrow 0 $ as $n\rightarrow \infty$ we have:
\begin{align*}
u^n(t,x)\quad\text{converges to}\quad u(t,x)\quad\text{in}\quad L^p([0,T] \times \Omega, W^{1,p}_{loc}(\mathbb{R}^{d}))
\end{align*}
as $n \to \infty$, where $u^n(t,x)$ and $u(t,x)$ are the unique $W^{1,p}-$weak solutions of the Cauchy problem \eqref{trasport} for the drift  $b_n$ and $b$, respectively.
\end{proposition}


\begin{proof} From Lemma \ref{renor} and Theorem \ref{uni}  we obtain  that
\begin{align*}
u^n(t,x)=u_0(X_{t}^{-1,n})\quad\text{and}\quad u(t,x)=u_0(X_{t}^{-1})\,.
\end{align*}
We also consider $u^{\varepsilon,n}(t,x)=u_0^{\varepsilon}(X_{t}^{-1,n}(x))$ the unique solution of the stochastic transport equation with vector field $b^n$ and initial condition $u_0^{\varepsilon}\in C_c^{\infty}(\R^d)$ such that $u_0^{\varepsilon}$ converges to $u_0$ in $W^{1,p}(\R^d)$. We observe that $u_0^{\varepsilon}\in W^{1,p'}(\R^d)$ with $p'>d$, then by Morrey's inequality $u_0^{\varepsilon} \in C_b^{0,\beta}(\R^d)$, where $\beta=1-\frac{d}{p'}$.  In this way, if $K$ is a compact in $\R^d$ we have
\begin{align*}
\int_{0}^{T} & \int_{K} \mathbb{E}|u^n(t,x)-u(t,x)|^{p} dx ds \nonumber\\[5pt]
& \leq C \int_{0}^{T} \int_{\Omega}\int_{K} | u_0(X_{t}^{-1,n}(x))- u_0^{\varepsilon}(X_{t}^{-1,n}(x))|^{p} dx \ \mathbb{P}(d\omega) \ dt  \nonumber\\[5pt]
& \quad+C \ \int_{0}^{T} \int_{\Omega}\int_{K} | u_0^{\varepsilon}(X_{t}^{-1,n}(x))- u_0^{\varepsilon}(X_{t}^{-1}(x))|^{p} dx \ \mathbb{P}(d\omega) \ dt \nonumber\\[5pt]
& \quad+C \ \int_{0}^{T} \int_{\Omega}\int_{\R^d} | u_0^{\varepsilon}(X_{t}^{-1}(x))- u_0(X_{t}^{-1}(x))|^{p} dx \ \mathbb{P}(d\omega) \ dt \nonumber\\[5pt]
& \leq C \ \int_{K} | u_0(y)- u_0^{\varepsilon}(y)|^{p} dy \nonumber\\[5pt]
& \quad+ C \ \int_{0}^{T} \int_{\Omega}\int_{K} | X_{t}^{-1,n}(x))-X_{t}^{-1}(x)|^{\beta p} dx \ \mathbb{P}(d\omega) \ dt \nonumber\\[5pt]
& \quad + C \ \int_{\R^d} |u_0^{\varepsilon}(y)- u_0(y)|^{p} dy \,.
\end{align*}
Thus, from (\ref{est2}), doing first $\varepsilon \to 0 $ and after $n\to \infty$, we get that $u^{n}\rightarrow u$ in $L^{p}(\Omega\times [0,T], L_{loc}^{p}(\R^{d}))$. Now, we see that
\begin{align}\label{eq-last}
\int_{0}^{T} \int_{K} & \mathbb{E}[|Du^n(t,x)-Du(t,x)|^{p} ]dx ds \nonumber\\[5pt]
& \leq \int_{0}^{T} \int_{K} \mathbb{E}[| Du_0(X_{t}^{-1,n}(x))DX_{t}^{-1,n}(x) - Du_0(X_{t}^{-1}(x))DX_{t}^{-1}(x)|^{p}]   dx ds \nonumber\\[5pt]
& \leq C \int_{0}^{T} \int_{K} \mathbb{E}[| Du_0(X_{t}^{-1,n}(x)) - Du_0^{\varepsilon}(X_{t}^{-1,n}(x))|^{p} \ |DX_{t}^{-1,n}(x)|^p] \ dx \ dt  \nonumber\\[5pt]
& \quad + C \ \int_{0}^{T} \int_{K} \mathbb{E}[|Du_0^{\varepsilon}(X_{t}^{-1,n}(x)) - Du_0^{\varepsilon}(X_{t}^{-1}(x))|^{p} \ |DX_{t}^{-1,n}(x)|^p] \ dx \ dt  \nonumber\\[5pt]
& \quad + C \ \int_{0}^{T} \int_{K} \mathbb{E}[|Du_0^{\varepsilon}(X_{t}^{-1}(x))|^p \ |DX_{t}^{-1,n}(x) - DX_{t}^{-1}(x)|^{p} ] \ dx  \ dt  \nonumber\\[5pt]
& \quad + C \ \int_{0}^{T} \int_{K} \mathbb{E}[|Du_0^{\varepsilon}(X_{t}^{-1}(x))-Du_0(X_{t}^{-1}(x))|^p \ |DX_{t}^{-1}(x)|^{p} ] \ dx  \ dt \nonumber\\[5pt]
& \leq C \int_{0}^{T} \int_{K} | Du_0(y) - Du_0^{\varepsilon}(y)|^{p} \ \mathbb{E}[|DX_{t}^{-1,n}(X_{t}^{n}(y),\omega)|^p] \ dy \ dt  \nonumber\\[5pt]
& \quad + C \ \int_{0}^{T} \int_{K} \mathbb{E}[|X_{t}^{-1,n}(x) - X_{t}^{-1}(x)|^{\beta p} \ |DX_{t}^{-1,n}(x)|^p] \ dx \ dt  \nonumber\\[5pt]
&\quad + C  \bigg(\int_{0}^{T} \int_{K} \mathbb{E}[|Du_0^{\varepsilon}(X_{t}^{-1}(x))-Du_0(X_{t}^{-1}(x))|^{p p''}]  \ dx  \ dt\bigg)^{\frac{1}{p''}}  \nonumber\\[5pt]
& \quad\quad\quad \times \bigg(\ \int_{0}^{T} \int_{K}\mathbb{E}[|Du_0^{\varepsilon}(X_{t}^{-1}(x))|^q ] \ dx  \ dt \bigg)^{\frac{p}{q}} \nonumber\\[5pt]
& \quad + C \ \int_{0}^{T} \int_{K} |Du_0^{\varepsilon}(y)-Du_0(y)|^p \ \mathbb{E}[ |DX_{t}^{-1}(X_t(y),\omega)|^{p} ] \ dy  \ dt \,,
\end{align}
where $\displaystyle \frac{p}{q} + \frac{1}{p''}=1$. Thus, considering \eqref{eq-flow-flow}, \eqref{eq-cofator}, the estimations \eqref{est1}, \eqref{est2} and \eqref{est3}, and doing first $\varepsilon \to 0$ and after $n\to \infty$ we have $D u^{n}\rightarrow D u$ in $L^{p}(\Omega\times [0,T], L_{loc}^{p}(\R^{d}))$. Therefore, by the calculus made above the proof is complete.
\end{proof}

\section*{Acknowledgements}

I would like to thank the anonymous referee for your comments which help me improve this work. I also would like to thank the colleague Christian Olivera for his important suggestions in order to improve this paper.


\end{document}